\documentclass[11pt]{amsart}
\usepackage{amsmath}
\usepackage{amssymb}
\usepackage{graphicx}
\usepackage{a4wide}

\newcommand{\X}{\ensuremath{\mathbb{X}}}
\newcommand{\Z}{\ensuremath{\mathbb{Z}}}
\newcommand{\R}{\ensuremath{\mathbb{R}}}

\renewcommand{\H}{\ensuremath{\mathbb{H}}}
\newcommand{\T}{\ensuremath{\mathcal{T}}}
\renewcommand{\S}{\ensuremath{\mathbb{S}}}
\newcommand{\F}{\ensuremath{\mathcal{F}}}

\DeclareMathOperator{\supp}{supp}
\DeclareMathOperator{\conv}{conv}

\newtheorem{thm}{Theorem}[section]
\newtheorem{lemma}[thm]{Lemma}
\newtheorem{cor}[thm]{Corollary}

\newtheorem{defi}[thm]{Definition}

\begin{document}

\title{The Lonely Vertex Problem}

\author{D. Frettl\"oh}
\address{Fakult\"at f\"ur Mathematik, Universit\"at Bielefeld,
Postfach 100131, 33501 Bielefeld,  Germany}
\email{dirk.frettloeh@math.uni-bielefeld.de}
\urladdr{http://www.math.uni-bielefeld.de/baake/frettloe}

\author{A. Glazyrin}
\address{Moscow State University, Leninskie Gory, 119992 Moscow GSP-2, 
  Russia}
\email{xoled@rambler.ru}

\begin{abstract}
In a locally finite tiling of $\R^n$ by convex polytopes, each point
$x \in \R^n$ is either a vertex of at least two tiles, or no vertex at
all. 
\end{abstract}

\maketitle

\section{Introduction} 
In \cite{fr}, the following problem was stated in the context of
finite local complexity of self-similar substitution tilings, see
Section \ref{apps} for details. Throughout the text, 'vertex' always
means the vertex of a convex polytope in the usual geometric sense,
see for instance \cite{z}. It means neither a combinatorial vertex of
a tile, nor the vertex of a tiling in the sense of \cite{gs} (that is,
an isolated point of the intersection of finitely many tiles of a
tiling).   

{\bf Question 1:}
In a locally finite tiling $\T$ of $\R^n$, where all tiles are convex
polytopes, is there a point $x$ which is the vertex of exactly one tile?

In other words: Is there a 'lonely vertex' in a locally finite polytopal
tiling?

For tilings in dimension $n=1$ and $n=2$, it is easy to see that the
answer is negative. In the sequel we show that the answer is negative
for all dimensions $n$. In the remainder of this section we will
fix the notation and discuss the necessity of the requirement 'locally
finite'. In Section 2 we
obtain the main results, namely, Theorem \ref{lincomb}, Theorem
\ref{sphere}, and the answer to Question 1 in Theorem \ref{lvert}. In
Section 3 we apply these results to prove a condition for local finite
complexity of self-similar substitution tilings with integer factor. 
Section 4 contains some further remarks.

Let $\R^n$ denote the $n$-dimensional Euclidean space. The
$n$-dimensional unit sphere is denoted by $\S^n$. For two points $x,y
\in \R^n$, the line segment with endpoints $x$ and $y$ is denoted by
$\overline{xy}$. A
(convex) {\em polyhedron} is the intersection of finitely many
closed halfspaces. A (convex) {\em polytope} is a bounded polyhedron.
In the following, only convex polytopes are considered. Thus we drop
the word 'convex' in the sequel, the term 'polytope' always means
convex polytope. A spherical polytope is the intersection of a sphere
with centre $x$ with finitely many halfspaces $H_i$, where $x \in 
\bigcap_i H_i$.

Let $\X$ be either a Euclidean or a spherical space. A collection of
polytopes $\T=\{T^{}_n\}^{}_{n \ge 0}$ which is a covering of
$\X$ --- that is, the union of all polytopes $T_i$ equals 
$\X$ --- as well as a packing of $\X$ --- that is, the interiors of the
polytopes are pairwise disjoint --- is called a (polytopal) \emph{tiling}.
A tiling $\T$ is called {\em locally finite} if each bounded set 
$U \in \X$ intersects only finitely many tiles of $\T$.
%; that is: $U \cap T = \varnothing$ for all but finitely many $T \in \T$.

If we do not require the tiling to be locally finite, lonely vertices
are possible. For instance, consider a tiling in $\R^2$ which 
contains the following tiles (see Figure \ref{nonlfex}): A rectangle
$R$ with vertices $(1,0), \; (-1,0), \; (-1,-1), \; (1,-1)$, a square
$S$ with vertices $(0,0), \; (0,1), \; (-1,1),\; (-1,0)$, and
rectangles $T_k$ with vertices $(\frac{1}{2^k},0), \; (\frac{1}{2^k},1), \;
(\frac{1}{2^{k+1}},1)\;
(\frac{1}{2^{k+1}},0),$, where $k \ge 0$. Such a tiling is obviously not
locally finite: each sphere with centre $(0,0)$ intersects infinitely 
many tiles. The tile $S$  has $(0,0)$ as a vertex, and $(0,0)$ is
vertex of no other tile. That means, such a tiling contains a lonely
vertex at $(0,0)$. The requirement of local finiteness is therefore 
necessary.

\begin{figure}
\includegraphics[width=120mm]{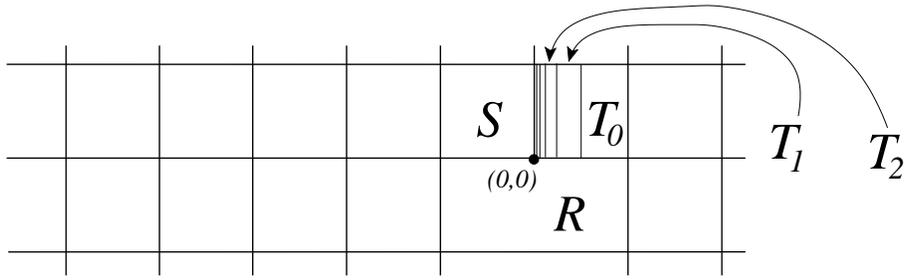}
\caption{\label{nonlfex} A lonely vertex at $(0,0)$ in a tiling which
  is not locally finite.}
\end{figure}

\section{The main result} \label{main}

We say that a polytope $P$ and a hyperplane $H$ are {\em just
  touching} if $P \cap H \ne \varnothing$, but $\mbox{int}(P) \cap H
= \varnothing$, where $\mbox{int}(P)$ denotes the interior of $P$.
We define the indicator $\textbf{I}_P$ for the convex
$n$-dimensional spherical polytope $P$ as the function that equals
1 in all internal points of $P$ and 0 else. In what follows we
say that two functions are equal if they are equal in all
points except in a set of Lebesgue measure zero. We call a convex
$n$-dimensional spherical polytope a {\em B-type polytope} if it
contains two ends of some diameter of the sphere, and an {\em A-type
polytope} else.

\begin{thm} \label{lincomb}
The indicator of any A-type polytope cannot be equal to the linear
combination of indicators of a finite number of B-type polytopes.
\end{thm}

\begin{proof} We will prove this theorem by induction on the dimension
$n$ of the embedding space $\R^n \supset \S^{n-1}$.  Base of 
induction: $n=1$. This case is obvious because the unit sphere in
$\mathbb{R}^1$, namely, $\S^0 = \{-1,1\}$, is the only B-type polytope
in $\S^0$. 

The step of induction is much more demanding. Let Theorem
\ref{lincomb} be true for all dimensions less than $n$. We assume that
it's false for $n$. So there is one A-type polytope $P$ and $k$ B-type
polytopes $Q_1,\ldots,Q_k$ such that
\[\textbf{I}_P-\sum_1^k \alpha_i \textbf{I}_{Q_i}=0 \]
for some $\alpha_i \in\R$. Consider
any $(n-1)$-dimensional hyperplane containing the centre $x$ of
the sphere, for instance $\{ x_1=0 \}$. For $f:\mathbb{R}^n \rightarrow
\mathbb{R}$ we define $f_0^+$ and $f_0^-$:
\[f_0^+(x_2,\ldots,x_n)=\lim_{m \rightarrow
\infty}f(\frac 1 {2^m}, x_2,\ldots,x_n)\]
\[f_0^-(x_2,\ldots,x_n)=\lim_{m
\rightarrow \infty}f(-\frac 1 {2^m}, x_2,\ldots,x_n)\] if these
limits exist. Let $T$ be an $n$-dimensional spherical polytope,
and let $T_0=T \cap \{x_1=0 \}$. $T_0$ is a spherical polytope of
lesser dimension.

\begin{lemma}
For $f=\textbf{I}_T$, the function $f_0^+$ exists in all points of
$\mathbb{R}^{n-1}$. Moreover, $f_0^+=\textbf{I}_{T_0}$ holds if not
all the internal points of $T$ are lying in negative semispace,
and $f_0^+=0$ else.
\end{lemma}

\begin{proof} There are three cases: the interior of $T$ intersects
$\{ x_1=0 \}$, or $T$ just touches this hyperplane and lies in the positive
semispace, or $T$ just touches this hyperplane and lies in the negative
semispace. All these cases are rather obvious.
\end{proof}

The same lemma is true for $f_0^-$.

Let us consider now $f=\textbf{I}_P-\sum_1^k \alpha_i
\textbf{I}_{Q_i}$. Without loss of generality, let one of the
$(n-1)$-dimensional faces of $P$ be contained in $\{ x_1=0 \}$, and
let $P$ lie in the positive semispace. Then $f_0^+$ exists, and
$f_0^+=\textbf{I}_{P_0}-\sum_1^k \alpha_i \textbf{I}_{Q_i^+}$,
where $Q_i^+ = Q_i \cap \{x_1=0\}$ if not all the internal points
of $Q_i$ are lying in negative semispace, and $Q_i^+ =
\varnothing$ else. Likewise, $f_0^-=-\sum_1^k \alpha_i
\textbf{I}_{Q_i^-}$, where $Q_i^-$ are defined analogously.
Obviously $f_0^+=0$ and $f_0^-=0$ holds, because they are limits of
sequences which are equal to 0. We define $g=f_0^+-f_0^-$. It
follows that $g=0$ and $g=\textbf{I}_{P_0}-\sum_1^k \alpha_i
(\textbf{I}_{Q_i^+}-\textbf{I}_{Q_i^-})$. If the interior of $Q_i$
intersects $\{ x_1=0 \}$, then $Q_i^+=Q_i^-$, and the 
corresponding brackets in the sum are equal to 0. (At this point 
convexity is required.) If 
$Q_i$ just touches this hyperplane, then one of the members 
in the corresponding term in brackets is equal to 0. So
\[0=\textbf{I}_{P_0}-\sum_1^k \beta_i \textbf{I}_{S_i},\]
where $S_i=\varnothing$ if the interior of $Q_i$ intersects the
hyperplane, $S_i=Q_i^+$ and $\beta_i=\alpha_i$ if $Q_i$ just
touches the hyperplane and lies in the positive semispace, $S_i=Q_i^-$
and $\beta_i=-\alpha_i$ if it just touches the hyperplane and lies
in the negative semispace.

\begin{lemma}
If a B-type polytope $Q$ just touches a hyperplane $H$ through the
centre $x$ of a sphere, % $\dim(H \cap Q)=n-1$, 
then the polytope $Q \cap H$ is also a B-type polytope.
\end{lemma}

\begin{proof} Any B-type polytope contains two ends of some diameter
of the $n$-sphere, say, points $k, \ell$. If $\overline{k \ell} \cap 
H =\{x\}$, then $H$ intersects the interior of the polytope $Q$.
This is impossible, since $Q$ and $H$ are just touching. 
Therefore $\overline{k \ell} \subset H$. Hence the
polytope $Q \cap H$ contains two ends of some diameter of the
sphere and is a B-type polytope.
\end{proof}

So all $S_i$ are B-type polytopes, and $P_0$ is an A-type
polytope. We have a contradiction with the proposition of the
induction. This completes the proof of Theorem \ref{lincomb}.
\end{proof}

\begin{thm} \label{sphere}
Any sphere $S$ in $\mathbb{R}^n$ cannot be partitioned in B-type
polytopes and exactly one A-type polytope.
\end{thm}

\begin{proof} We assume there is such a decomposition. $P$ is 
an A-type polytope and $Q_1,\ldots,Q_k$ are B-type polytopes. 
Let $M_1$ and $M_2$ are two hemispheres such that 
$M_1 \cup M_2 = S$. Then
\[ \textbf{I}_P+\sum_1^k \textbf{I}_{Q_i} -\textbf{I}_{M_1}-\textbf{I}_{M_2}=0.\] 
This contradicts Theorem \ref{lincomb}.
\end{proof}

\begin{thm} \label{lvert}
Let $\T$ be a locally finite polytopal tiling in $\mathbb{R}^n$.
There is no point $x \in \mathbb{R}^n$ such that $x$ is a vertex
of exactly one polytope of $\T$.
\end{thm}

\begin{proof} We choose a sphere $S$ with centre $x$ such that all
faces of the polytopes of $\T$ intersecting $S$ 
contain $x$. We can find such a sphere since $\T$ is locally
finite. If $x$ is a vertex of a tile $T$ in $\T$, then its
intersection with $S$ is an A-type polytope. If $x \in T$ is not a
vertex, then the intersection $T \cap S$ is a B-type polytope. 
Because of Theorem \ref{sphere} there can't be exactly one A-type 
polytope. So $x$ can't be a vertex of exactly one polytope of the
tiling $\T$.   
\end{proof}

{\bf Remark:} The last result generalizes immediately to spherical and
hyperbolic tilings: Even though no two of Euclidean space $\R^n$,
hyperbolic space $\H^n$ and spherical space ${\mathbb S}^n$ are
conformal to each other, they are {\em locally conformal}: There is a
map $f_x: \X \to \X'$ (where $\X, \X' \in \{ \R^n, \H^n, {\mathbb S}^n
\}$), such that, for a given point $x \in \X$, lines through $x$ are
mapped to lines through $f_x(x)$, and their orientations and the angles
between such lines are preserved. This is all we need to generalize
the result.

\begin{cor}
Each $k$-face of some tile in a locally finite $\T$ tiling of $\R^n$
by polytopes is covered by finitely many $k$-faces of some other
tiles. 
\end{cor}

\begin{proof}
We use induction on $k$. The case $k=0$ is Theorem \ref{lvert}: 
Any vertex is covered by a vertex of some other tile.

Let the statement be true for $k-1$. Let $F$ be a $k$-face of 
some tile $T \in \T$. Let $x$ be a point in the relative interior 
of $F$. As above, let $S$ be a sphere with centre $x$
such that 

\begin{center}
(A) \ All faces of polytopes in $\T$ intersecting $S$ contain $x$. 
\end{center}

Since $F$ is a $k$-face, $F' = F \cap S$ is a $(k-1)$-face 
of $T \cap S$ (in the spherical tiling $\T \cap S$). 
By the proposition of induction, $F'$ is covered by $(k-1)$-faces $F_i$. 
Because of (A), the convex hull $\conv(x,F')$ of $x$ and $F'$ in
$\R^n$ is covered by $\conv(x,F_i)$, which are subsets of $k$-faces in
$\T$. This is true for any $x$ in the relative interior of
$F$, thus everywhere. Because of local finiteness, $F$ is covered
by finitely many $k$-faces.
\end{proof}

The following theorem is used in the next section.

\begin{thm} \label{graph}
Given a polytopal tiling $\T$, let $G=(V,E)$ be the following undirected 
graph: $V$ is the set of all vertices of tiles in $\T$. Vertices are
identified if they are equal as elements of $\R^n$. $E$ is the set of
edges in $G$, where $(x,y) \in E$ iff the line segment $\overline{xy}$
is an entire edge of some tile in $\T$. 
Then, all connected components of $G$ are infinite.
\end{thm}

\begin{proof} 
Obviously, any two vertices of some polytopal tile $T$ are connected
by a finite path of edges of $T$, so they are in the same 
connected component of $G$. Therefore, each tile belongs either 
entirely to a connected component of $G$ or not.

Assume there is a finite connected component $C$ in $G$. Let $\F$ be
the set of all tiles belonging to $C$. Being finite, the union
$\supp(\F)$ (which is a polytope, though not necessarily convex) 
has some outer vertex $x$. 

The vertex $x$ corresponds to an A-type polytope as above. By
Theorem \ref{sphere}, there is at least one further A-type polytope, 
belonging to a tile $T \notin \F$. Because $T$ contributes an $A$-type 
polytope, $x$ is a vertex of $T$. This contradicts $T \notin \F$,
proving the claim. 
\end{proof}

\section{Application to substitution tilings} \label{apps}

The discovery of nonperiodic structures with long range
order (for instance, Penrose tilings and quasicrystals) had a large
impact to many fields in mathematics, see for instance \cite{lag}. 
Tile-substitutions are a simple and powerful tool to generate 
interesting nonperiodic
structures with long range order, namely: substitution tilings. The
basic idea is to give a finite set of {\em prototiles}
$T_1, \ldots, T_m$, together with a rule how to enlarge each prototile
by a common {\em inflation factor} $\lambda$ and then dissect it into
--- or more general, replace it by --- copies of the
original prototiles. Figure \ref{substbsp} shows some examples of
substitution rules. Note, that a substitution $\sigma$ maps tiles to
finite sets of tiles, finite sets of tiles to (larger) finite sets of
tiles, and tilings to tilings. 
By iterating the substitution rule, increasingly
larger portions of space are filled, yielding a tiling of the entire
space in the limit. For a more precise definition of substitution
tilings, see for instance \cite{fr2}. For a collection of substitution
tilings, and a glossary of related terminology, see \cite{frh}.

\begin{figure}
\includegraphics{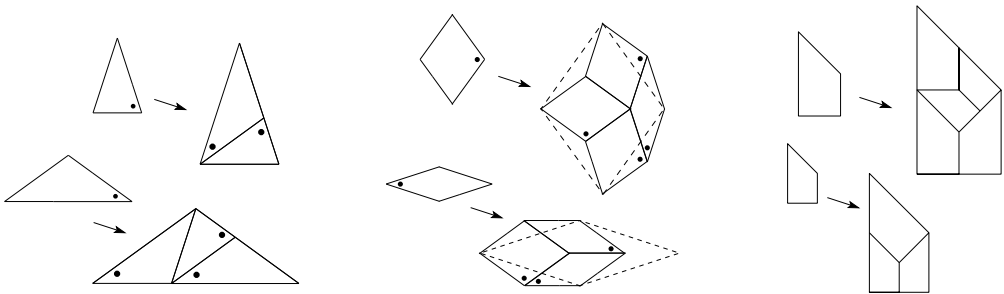}
\caption{\label{substbsp} Three examples of tile-substitutions: The
  Penrose substitution rule for triangular tiles (left), the
  substitution rule for binary tilings (centre), the semi-detached
  house substitution rule (right).} 
\end{figure}

A tile-substitution rule with a proper dissection, that is, where
\begin{equation} \label{selfsim}
 \lambda T_i = \bigcup_{T \in \sigma(T_i)} T \qquad  (1 \le i \le m)
\end{equation}
(where the union is non-overlapping) is called {\em self-similar} 
tile-substitution. If \eqref{selfsim} does
not hold, as in Figure \ref{substbsp} (centre), one may still speak of
a substitution tiling, but not of a self-similar tiling.

The following definition turned out to be useful in the theory of
nonperiodic tilings. It rules out certain pathological cases and
is consistent with other concepts within this theory, for instance the
{\em tiling space}, or the {\em hull} of a tiling \cite{sol}, \cite{kp}.  

\begin{defi}
Let $\sigma$ be a tile-substitution with prototiles $T_1, \ldots,
T_m$. The sets $\sigma^k(T_i)$ are called {\em ($k$-th order) 
supertiles}.\\ A tiling $\T$ is called {\em substitution tiling} (with
tile-substitution $\sigma$) if for each finite subset $\F \subset \T$
there are $i,k$ such that $\F$ is congruent to a subset of some
supertile $\sigma^k(T_i)$. \\
The family of all substitution tilings with tile-substitution $\sigma$
is denoted by $\X_{\sigma}$. 
\end{defi}

Many results in the theory of substitution tilings require the tilings
under consideration to be of finite local complexity, compare
for instance \cite{sol}, \cite{sol2}, \cite{lms}.

\begin{defi}
A tiling $\T$ has {\em finite local complexity (FLC)} if for each
$r>0$ there are only finitely many different constellations of
diameter less than $r$ in $\T$, up to translation. 
\end{defi}
Usually, if a
certain substitution tiling has FLC, this is easy to see. For
instance, each vertex-to-vertex tiling with finitely many prototiles
has FLC. More general, the following condition is frequently used
\cite{fr}.  

\begin{lemma} \label{flclemma}
A tiling is FLC iff there are only finitely many different
constellations of two intersecting tiles, up to translation. 
\end{lemma}

On the other hand, if a tiling does not 
have FLC, this can be hard to prove, see \cite{dan,f-r}. The
following theorem covers a broad class of substitution tilings where
the inflation factor $\lambda$ is an integer number. An example of
such a tile-substitution is shown in Figure \ref{substbsp} (right),
where the inflation factor is 2. 
A weaker version of this theorem was proved in \cite{fr}, and it was
realized that a negative answer to Question 1 would yield a stronger
result. Thus Question 1 was stated in \cite{fr} as an open problem.

\begin{thm} \label{flc}
Let $\T$ be a self-similar substitution tiling with polytopal
prototiles and integer inflation factor. Without loss of
generality, let 0 be a vertex of each prototile. If the \Z-span of all
vertices of the prototiles is a discrete lattice, then $\T$ is of
finite local complexity. 
\end{thm}

It is remarkable that a requirement on the shape of the prototiles,
without any word about the tile-substitution itself, suffices to
guarantee FLC. Note, that we do not require the tiles to be convex at
this point. It suffices that they are unions of finitely many convex
polytopes. 

\begin{proof}
We begin by showing that all vertices contained in some supertile 
$S = \sigma^k(T_i) = \{T, T', T'', \ldots \}$ 
belong to the same connected component of the graph $G$, 
with $G$ as in Theorem \ref{graph}. First we consider vertices on the
edge of the support of a supertile. A (super-)edge of the supertile $S$
consists of entire edges of some tiles. Thus, all vertices in a
single (super-)edge of $S$ belong to the same component $C$ of
$G$. Consequently, all vertices in the union of the 
edges of the supertile $S$ belong to $C$. 

Now, consider a $k$-face $F$ of $S$, where $k\ge 2$. Let 
all vertices on the boundary of $F$ (of dimension $k-1$) be in the
same component $C$ of $G$. If there is a vertex $x$ in $F$ with 
$x \notin C$, it belongs to a finite component of $G$ in $F$ which is
disjoint with the boundary of $F$. Thus, $F$ can be extended to a
$k$-dimensional polytopal tiling with the finite component $C$ in the
corresponding graph $G$. But this contradicts Theorem \ref{graph}.  
Consequently, all vertices in $F$ belong to $C$. Inductively --- by
finite induction on $k$ --- all vertices contained in the supertile 
$S$ belong to $C$.  

Now, let $\Gamma$ be the lattice spanned by the vertices of the
prototiles. Since the inflation factor is an integer, the vertices
of each supertile $S$ are elements of $\Gamma$. All tile-vertices
contained in $S$ belong to the same connected component of $G$, thus
 --- by definition of $G$ --- they are connected by a finite path of
entire tile edges $\overline{xy}$ with some vertex of $S$. By the
condition in the theorem, $x-y \in \Gamma$ for all such edges
$\overline{xy}$. Therefore, all vertices in the supertile are
contained in $\Gamma$. Consequently, all  vertices of $\T$ are
elements of $\Gamma$.  

In particular, if two tiles in $\T$ have nonempty intersection, there
is only a finite number of possible position of the vertices of these
tiles, by the discreteness of $\Gamma$. By Lemma \ref{flclemma}, $\T$
has FLC.
\end{proof}

\section{Remarks}

We have established the impossibility of a lonely vertex in a locally 
finite polytopal tiling in Euclidean, spherical and hyperbolic space of 
any dimension. Some consequences are discussed in this paper.
Naturally, further questions arise. For instance, what can be said
about lonely vertices in locally finite tilings with non-convex tiles?

Another natural question is: What can be said about exactly two 
vertices? Since a lonely vertex is impossible, there may be 
restrictions for constellations around a point which is a vertex 
of exactly two tiles $T, T'$. Indeed, one obtains the following result.
Roughly spoken, it means that edges of $T$ and $T'$ either are
coincident or opposite. In particular, the number of edges of
$T$ containing $x$ equals the number of edges of
$T'$ containing $x$. For clarity, we state the result in terms of
A-type and B-type polytopes.

\begin{thm}
Let a locally finite tiling of the unit sphere $\S^n$ by polytopes 
contain exactly two A-type polytopes $P, P'$. Let $x$ be a 
vertex of $P$. Then either $x$ or $-x$ is a vertex of $P'$.
\end{thm}

\begin{proof}
The cases $n=0$ and $n=1$ are obvious. So, let $n>1$. 

We proceed by considering possible shapes of B-type polytopes. Any
B-type polytope is cut out of the unit sphere $\S^n$, embedded in
$\R^{n+1}$, by halfspaces $H^+_1, \ldots, H^+_m$,
where $x \in \bigcap_i H^+_i$. Each such halfspace
$H^+_i$ can be represented by a vector $c_i$ which is normal
to the bounding hyperplane $H_i = \partial H^+_i$: 
$H^+_i = \{ x : c_i x \ge 0 \}$. We can assume the set of
hyperplanes to be minimal. That is, the normal vectors of these
hyperplanes are linearly independent (otherwise there would be a
superfluous defining inequality $c_i x \ge 0$; that means, a
superfluous halfspace). Therefore, the intersection 
$M:= \bigcap_i H_i$ is an $(n+1-m)$-dimensional linear subspace.
Since the considered polytope is B-type, it contains two endpoints of
some diameter of the sphere. Thus $M$ has to be at least of 
dimension one. It follows  $m \le n$, and the intersection 
$\S^n \cap M$ (which is the boundary of the considered B-type 
polytope), is an $(n-m)$-dimensional unit sphere. In particular, 
a B-type polytope has a vertex $x$ if and only if it is defined by 
exactly $n$ halfspaces. Then, $-x$ is also a vertex of this 
B-type polytope.

By Theorem \ref{lvert}, the vertex $x$ of $P$ is a vertex of some
further polytope. Either A-type  (then $P'$), or B-type, say,
$P''$. In the latter case, by the reasoning above, $-x$ is a vertex of
$P''$, too.  

If $-x$ would be surrounded entirely by B-type polytopes, $x$ also
would, which is impossible. Thus, $-x$ is a vertex of an A-type
polytope. The only possibility is that $-x$ is a vertex of $P'$. 
\end{proof}

\section*{Acknowledgments}
It is a pleasure to thank Nikolai Dolbilin, Alexey Tarasov and in
particular Alexey Garber for valuable discussions.
D.F.\ acknowledges support by the German Research Council (DFG) within
the Collaborative Research Centre 701.


\begin{thebibliography}{ZZZZ9}


\bibitem[\textsc{D}]{dan}
  L.\ Danzer:
  Inflation species of planar tilings which are not of locally finite
  complexity,
  {\em Proc.\ Steklov Inst.\ Math.} 239 (2002) 118-126.

\bibitem[\textsc{FrR}]{f-r}
  N.P.\ Frank, E.A.\ Robinson, Jr.:
  Generalized beta-expansions, substitution tilings, and local finiteness,
  to appear in Trans.\ Amer.\ Math.\ Soc.

\bibitem[\textsc{F}]{fr}
  D.\ Frettl\"oh:
  Nichtperiodische Pflasterungen mit ganzzahligem Inflationsfaktor,
  Ph.D.\ Thesis, Dortmund (2002);
  \texttt{http://hdl.handle.net/2003/2309}.

\bibitem[{\sc F2}]{fr2}
  D.\ Frettl\"oh:\
  Duality of model sets generated by substitutions,
  {\em Rev.\ Roumaine Math.\ Pures Appl.} {\bf 50} (2005) 619-639;
  {\tt math.MG/0601064}.

\bibitem[{\sc FH}]{frh}
  D.\ Frettl\"oh, E.\ Harriss:
  Tilings Encyclopedia, available online at:\\
  \texttt{http://tilings.math.uni-bielefeld.de.}

\bibitem[{\sc GS}]{gs} 
  B.\ Gr\"unbaum, G.C.\ Shephard: 
  {\it Tilings and Patterns}, Freeman, New York (1987).

\bibitem[\textsc{KP}]{kp}
  J.\ Kellendonk and I.F.\ Putnam:
  Tilings, $C^{\ast}$-algebras and $K$-theory,
  in: \textit{Directions in Mathematical Quasicrystals},
  M.\ Baake and R.V.\ Moody (eds.),
  CRM Monograph Series, vol.\ \textbf{13},
  AMS, Providence, RI (2000) pp.\ 177-206.

\bibitem[\textsc{Lag}]{lag}
  J.\ C.\ Lagarias:
  The impact of aperiodic order on mathematics,
  {\em Materials Science \& Engineering} A, 294--296 (2000) 186-191.

\bibitem[\textsc{LMS}]{lms}
  J.-Y.\ Lee, R.V.\ Moody and B.\ Solomyak:
  Consequences of pure point diffraction spectra for multiset
  substitution systems,
  \textit{Discrete Comput.\ Geom.} \textbf{29} (2003) 525-560.

\bibitem[\textsc{So}]{sol}
    B.\ Solomyak:
    Dynamics of self-similar tilings,
    \textit{Ergodic Theory Dynam. Systems} \textbf{17} (1997)
    695-738. \\
    B.\ Solomyak:
    Corrections to `Dynamics of self-similar tilings',
    \textit{Ergodic Theory Dynam.\ Systems} \textbf{19} (1999)
    1685.

\bibitem[{\sc So2}]{sol2} 
  B.\ Solomyak: 
  Non-periodicity implies unique composition property for
  self-similar translationally finite tilings, 
  {\it Discrete Comput.\ Geom.} {\bf 20} (1998) 265-279.

\bibitem[{\sc Z}]{z}
  G.\ Ziegler: {\it Lectures on Polytopes},
  Springer, New York (1995).

\end{thebibliography}
\end{document}